\documentclass[noams]{compositio}
\usepackage{amsmath}
\usepackage{amssymb}
\usepackage{epsfig}
\usepackage{color}
\usepackage{mathrsfs}
\usepackage[normalem]{ulem}

\usepackage[utf8]{inputenc}

\usepackage{tikz}
\usetikzlibrary{positioning}

\usepackage[margin=1in]{geometry}

\numberwithin{equation}{section}

\renewcommand{\contentsname}{Contents\\{\footnotesize\normalfont(A table
of contents should normally not be included)}}

\newif\ifdraft\drafttrue

\draftfalse

\newcommand {\pd}{\Phi_\Delta}

\newcommand\eq[2]{{\ifdraft{\ \tt [#1]}\else\ignorespaces\fi}\begin{equation}\label{eq:
#1}{#2}\end{equation}}

\newcommand {\equ}[1] {\eqref{eq: #1}}

\newcommand {\comm}[1] {\textcolor{orange}{#1}}

\newcommand {\wmn}{{W}_{m,n}}
\newcommand {\dmn}{{D}_{m,n}}
\newcommand {\admn}{\widehat{D}_{m,n}}

\newcommand{\Q}{{\mathbb {Q}}}

\newcommand{\R}{{\mathbb{R}}}

\newcommand{\Z}{{\mathbb{Z}}}

\newcommand{\re}{\operatorname{Re}}
\newcommand{\im}{\operatorname{Im}}

\newcommand{\N}{{\mathbb{N}}}

\newcommand{\vv}{{\bf{v}}}
\newcommand{\vb}{{\bf{b}}}

\newcommand{\y}{{\bf{y}}}

\newcommand{\vw}{{\bf{w}}}

\newcommand{\SL}{\operatorname{SL}}
\newcommand{\ASL}{\operatorname{ASL}}

\newcommand{\ggm}{G/\Gamma}

\newcommand{\SO}{\operatorname{SO}}

\newcommand{\diag}{{\operatorname{diag}}}

\newcommand{\df}{{\, \stackrel{\mathrm{def}}{=}\, }}

\newcommand{\x}{{\bf x}}

\newcommand{\vp}{{\bf p}}

\newcommand{\vq}{{\mathbf{q}}}

\newcommand{\T}{{\mathbb{T}}}

\newcommand{\vre}{\varepsilon}

\newcommand{\nz}{\smallsetminus\{0\}}

\newcommand\hd{Hausdorff dimension}

\newcommand\da{Diophantine approximation}

\newcommand\di{Diophantine}

\newcommand\hs{homogeneous space}

\newcommand{\ca}{{ A}}

\newcommand\dt{Dirichlet's Theorem}

\newcommand\mr{M_{m,n}
}
\newcommand\amr{$A\in M_{m,n}
$}

\newcommand {\ignore}[1] {}

\newcommand\ssm{\smallsetminus}

\newtheorem{thm}{Theorem}[section]

\newtheorem{lem}[thm]{Lemma}
\newtheorem{que}[thm]{Question}

\newtheorem{prop}[thm]{Proposition}

\newtheorem{cor}[thm]{Corollary}

\newtheorem{remark}[thm]{Remark}
\newtheorem{defn}[thm]{Definition}

\begin{document}

\title[An Inhomogeneous Dirichlet Theorem]{An Inhomogeneous Dirichlet Theorem \\ via Shrinking Targets 
}
\author{Dmitry Kleinbock}
\email{kleinboc@brandeis.edu}
\address{Brandeis University, Waltham MA, USA, 
02454-9110}
\author{Nick Wadleigh}
\email{wadleigh@brandeis.edu}
\address{Brandeis University, Waltham MA, USA, 
02454-9110}
%
\thanks{The first-named author was supported by NSF grants DMS-1101320 and DMS-1600814.}

\begin{abstract}
We give an integrability criterion on a real-valued non-increasing function $\psi$ guaranteeing that for almost all (or almost no) pairs $(A, \textbf{b})$, {where} $A$ {is} a {real $m\times n$ matrix} {and} ${\bf b} \in \R^m$, the system $$\|A \vq+{\bf b}-\vp\|^m< \psi({T})\hspace{10mm}\|\vq\|^n<{T}$$ is solvable in $\vp \in \Z^m$, $\vq \in \Z^n$ for all sufficiently large ${T}$. {The proof consists of a reduction to a shrinking target problem on the space of {grids} in $\R^{m+n}$.}  {We also comment on the homogeneous counterpart to this problem, whose $m=n=1$ case was recently solved, but whose general case remains open.}\end{abstract}

\maketitle


\section{Introduction and motivation}

\subsection{Homogeneous Diophantine approximation}   
\noindent Fix positive integers $m,n$. Let $M_{m,n}$ denote the space of real $m\times n$ matrices. The starting point for the present paper is the following theorem, proved by Dirichlet in 1842:

\begin{thm}[Dirichlet's Theorem]\label{dirichlet theorem} For any $A\in M_{m,n}$ and ${T}>1$, there exist $\vp\in \mathbb{Z}^m$, $\vq\in \mathbb{Z}^n\nz$ such that
\eq{dt}{ \|A\vq - \vp\|^m
\le \frac 1{{T}}
\ \ \ \mathrm{and}
\ \ \|\vq\|^n
< {T}.}
\end{thm}

{Here and hereafter $\| \cdot \|$ stands for the supremum norm on $\mathbb{R}^k$, $k\in\N$}. Informally speaking, a matrix $A$ represents a vector-valued function $\vq\mapsto A\vq$, and the above theorem asserts that one can choose a not-so-large nonzero integer vector $\vq$ so that the output of that function is close to an integer vector. In case $m=n=1$ the theorem just asserts that for any real number $\alpha$ and ${T}>1$, one of the first ${T}$ multiples of $\alpha$ lies within $1/{T}$ of an integer.
Theorem \ref{dirichlet theorem} is the archetypal \textit{uniform} diophantine approximation result, so called because it guarantees a non-trivial integer solution for \textit{all} ${T}$.  A weaker form of approximation (sometimes called \textit{asymptotic} approximation, see e.g.\ \cite{W, kim}) guarantees that such a system is solvable for an \textit{unbounded} set of ${T}$.  For instance, Theorem \ref{dirichlet theorem} implies {that} \equ{dt} is solvable for {an unbounded set of} ${T}$, \textit{a fortiori}.  The following corollary, which follows trivially from this weaker statement,  is the archetypal asymptotic result:

\begin{cor}\label{dirichlet corollary} For any $A\in M_{m,n}$ there exist infinitely many $\vq\in\Z^n$
such that
\eq{dc}{ \|A\vq - \vp\|^m {<} \frac1{\|\vq\|^n} \quad \text{for some }\vp\in\Z^m.}
\end{cor}

Together the aforementioned results {initiate} the {\sl metric theory of \da}, a field concerned with understanding sets of $A\in M_{m,n}$ which admit improvements to Theorem \ref{dirichlet theorem} and Corollary \ref{dirichlet corollary}.
  This paper 
{has been} motivated by an observation that the sensible ``first questions" about the asymptotic set-up were settled long ago, while the analogous questions about the uniform set-up remain open.
Let us start by reviewing what is known in the asymptotic set-up.

For a function $\psi: \R_+\to \R_+$, let us define
$\wmn(\psi)$, the set of {\sl $\psi$-approximable} matrices,
to be the set of $A\in M_{m,n}$ for which
there exist infinitely many $\vq\in\Z^n$ such that\footnote{This definition essentially coincides with the one given in \cite{KM} but differs slightly from other sources, {such as e.g. \cite[\S 13]{BDV}, where the inequality $\|A\vq - \vp\|{<} \|\vq\|\psi(\|\vq\|)$} is used instead of \equ{dcpsi}.}
\eq{dcpsi}{ \|A\vq - \vp\|^m\le \psi(\|\vq\|^n) \quad \text{for some }\vp\in\Z^m.}

\noindent Throughout the paper we use the notation $\psi_a(x) := x^{-a}$. Thus Corollary \ref{dirichlet corollary} asserts that $\wmn(\psi_1) = \mr$, and in the above definition we have simply replaced $ \psi_1(\|\vq\|^n) $ in \equ{dc} with $ \psi(\|\vq\|^n) $.  
Precise conditions for the Lebesgue measure of $\wmn(\psi)$ to be zero or full are given by

\begin{thm}[Khintchine-Groshev Theorem, {\cite{G}}]\label{khintchine theorem}
Given a non-increasing\footnote{The monotonicity condition can be removed unless $m=n=1$.} $\psi$,
the set $\wmn(\psi)$ has zero (resp.\ full) measure if and only if the series $\sum_k\psi(k)$ converges (resp.\ diverges). \end{thm}
\noindent {See \cite{Sp} or \cite{BDV} for details, and also \cite{KM} for an alternative proof  using dynamics on the space of lattices.}

\smallskip

{Questions related to similarly improving  Theorem \ref{dirichlet theorem} were first addressed in two seminal papers \cite{Davenport-Schmidt, Davenport-Schmidt2} by  Davenport and Schmidt. However}
no zero-one law analogous to Theorem \ref{khintchine theorem} has yet been proved in the set-up of {uniform approximation {for general
$m,n \in \N$}. Let us introduce the following definition: {for  a non-increasing  function $\psi:[T_0,\infty)\to \R_+$, where $T_0 > 1$ is fixed,} {say that $A\in M_{m,n}$ is {\sl $\psi$-$Dirichlet$}, or $A\in\dmn(\psi)$},
if} the system
\eq{dtpsi}{ \|A\vq - \vp\|^m
< \psi({T})
\ \ \ \mathrm{and}
\ \ \|\vq\|^n
< {T}}
has a nontrivial integer solution for all  large enough $T$. 
{In other words, we have replaced $ \psi_1({T}) $ in \equ{dt} with $ \psi({T})$, demanded the existence of nontrivial integer solutions for all ${T}$ except those belonging to a bounded set, and sharpened one of the inequalities in  \equ{dt}. The latter change, in particular, implies the following observation: for non-increasing $\psi$, membership in $\dmn(\psi)$ depends only on the solvability of the system \equ{dtpsi} at integer values of ${T}$. (To show this it suffices to replace $T$ with $\lceil T \rceil$ and use the monotonicity of $\psi$.)}

It is not difficult to see that ${D}_{1,1}(\psi_1)= \R$, and that for general $m,n$, almost every matrix is $\psi_1$-Dirichlet.
 {In contrast, it was proved} 
in \cite{Davenport-Schmidt2} for $\min(m,n) = 1$, and in \cite{KWe} for the general case, that  for any $c < 1$, the set $\dmn(c\psi_1)$ of $c\psi_1$-Dirichlet matrices has Lebesgue measure zero. This naturally motivates the following

\begin{que} \label{dirichlet question}
What is a necessary and sufficient condition on a non-increasing function $\psi$ (presumably expressed in the form of convergence/divergence of a certain series) guaranteeing that
the set $\dmn(\psi)$ has zero or full measure? \end{que}

 In \cite{KW} we give an answer to this question in case $m=n=1$, but in general Question \ref{dirichlet question} {seems to be} much harder than its counterpart for the sets $\wmn(\psi)$, answered by Theorem \ref{khintchine theorem}. We comment later in the paper on the reason for this difficulty, but the main subject of this paper is different: we take up an analogous inhomogeneous approximation problem, describe the analogues of the statements and concepts discussed in this section, and then show how an inhomogeneous analogue of Question \ref{dirichlet question} admits a complete solution  based on a correspondence between \da\ and dynamics on  {homogeneous spaces}.


\subsection{Inhomogeneous approximation: the main result}  
The theory of inhomogeneous \da\ starts when one replaces the values of a system of linear forms $A\vq$ by those of a system of {\sl affine forms} $\vq \mapsto A\vq + \textbf{b}$, where $A\in M_{m,n}$ and $\textbf{b}\in\R^m$. {Consider a non-increasing function $\psi:[T_0,\infty)\to \R_+$ and}, {following} the definition of the set $\dmn(\psi)$,
let us say {that} a pair $(A, \mathbf{b}) \in M_{m,n}\times \mathbb{R}^m$ is \emph{$\psi$-Dirichlet} if there exist $\vp \in \mathbb{Z}^m$, $\vq\in \mathbb{Z}^n$ such that
\eq{ideq}{\|A \vq+\mathbf{b}-\vp\|^m< \psi({T})\hspace{10mm}\|\vq\|^n< {T}}
whenever ${T}$ is large enough. {(Note that in this set-up there is no need to single out the case $\vq = 0$.)} Denote the set of $\psi$-Dirichlet pairs by $  \admn(\psi)$.  Note that, as {is the case} with $\dmn(\psi)$, membership {in} $\admn(\psi)$ depends only on the solubility of these inequalities at integer values of ${T}$, provided $\psi$ is non-increasing. {Hence without loss of generality one can assume $\psi$ to be continuous.}

Let us 
start  with the simplest case: $\psi \equiv c$ is a constant function, or $\psi= c\psi_0$ in our notation. It is a trivial consequence of Dirichlet's Theorem that whenever $c>0$, $$\|A \vq-\vp\|^m< c \hspace{10mm} \| \vq \|^n < {T}$$ is solvable in $ \vp \in \Z^m, \vq \in \Z^n\nz$ whenever ${T}> c^{-1}$. 
By contrast, it is clear that {one} cannot always {solve
$$\|A \vq+ \mathbf{b}-\vp\|^m< c \hspace{10mm} \| \vq \|^n < {T}$$}for $c\leq \frac{1}{2^{m}}$;
for example, take $A$ to be an integer matrix and take $\mathbf{b}$ with coordinates in $\Z +\frac12$. However, it follows from {Kronecker's} Theorem \cite[\S 3.5]{Cassels} that for a given $A\in M_{m, n}$, there exist $\mathbf{b}\in \mathbb{R}^m$ and $c>0$ such that $(A,\mathbf{ b}) \notin  \admn(c\psi_0)$ --- {which amounts to saying that $A\mathbb{Z}^n$ is not dense in $\mathbb{R}^m/\mathbb{Z}^m$} 
 --- only if $A^t(\mathbb{Z}^m\nz)$ contains an integer vector. The set of such $A$ has measure zero since it is the union over $\vq\in\mathbb{Z}^n$, $\vp\in \mathbb{Z}^m\nz$ of the sets $\{A: A^t \vp= \vq\}$. Thus for every $c>0$, $  \admn(c\psi_0)$ has full measure.

Once {$\psi$ is allowed} to decay to zero, {the sets $\admn(\psi)$ become smaller. 
In particular, using 
{dynamics} on the space of {grids in $\R^{m+n}$}, 
one can easily prove  (see Proposition \ref{c/t} below) that $\admn(C\psi_1)$ is null for any $C > 0$.
Thus {one can} naturally ask the following inhomogeneous analogue 
of Question \ref{dirichlet question}:}

\begin{que}\label{inhomdirichlet question}What is a necessary and sufficient condition on a non-increasing function $\psi$ (presumably expressed in the form of convergence/divergence of a certain series) guaranteeing that
the set $\admn(\psi)$ has zero or full measure?
\end{que}

The remainder of this work will be given to a proof of the following answer:

\begin{thm}\label{mainthm} 
{Given a non-increasing $\psi$,
the set $\admn(\psi)$ has zero (resp.\ full) measure if and only if the series \eq{psisum}{\sum_j \frac{1}{\psi(j)j^2}} diverges (resp.\ converges).}
\end{thm}


Note that this immediately gives results such as 
\begin{itemize}
\item[$\bullet$] $\admn({C}\psi_a)$ has zero (resp.\ full) measure if {$a\geq 1 $ (resp.\ $a<1$});
\item[$\bullet$]  for $\psi({T}) = {C}(\log {T})^b\psi_1({T})$,  $\admn(\psi)$ has zero (resp.\ full) measure if $b\leq 1$ (resp.\ $b>1$).
\end{itemize}

{Our argument is based on a correspondence between \da\ and homogeneous dynamics.
In the next section we {introduce the space of {grids} in $\R^{m+n}$  and reduce} the aforementioned inhomogeneous approximation problem 
{to} a shrinking target phenomenon for a flow on that space.  We do a warm-up problem, Proposition \ref{c/t}, that 
demonstrates the usefulness of {the reduction to dynamics} and introduces several key ideas to be used later.  This is followed by the statement of the main {dynamical result, Theorem {\ref{dynamicalmainthm}, which we prove in the two subsequent sections.} 
{The last section contains some concluding remarks, in particular a discussion of Question \ref{dirichlet question} and  other open questions.}

\ignore{Once we allow $\psi$ to decay to zero, we are guaranteed at least one $(\alpha, \beta)\notin D_{1,1}(\psi)$ with $\alpha \notin \mathbb{Q}$ [Ca, 3.3.3]. When we move to $\psi= c\psi_1$, we are able to compute the measure of $D_{m,n}(\psi)$ using the ergodicity of a diagonal flow on the space of {grids}.  We give the proof as a sort of warm-up that will motivate the dynamical system we will use to prove Theorem \ref{mainthm}.  It will also introduce some key ideas to be used later.  It will come after we set the stage with some notation and lemmas. }

\section{{Dynamics on the space of grids: a warm-up}}
} \label{dyn}

\noindent {Fix  $k\in\N$ and let {$${G_{k}} := \SL_k(\R)\text{ and } {\widehat{G}_{k}} := \ASL_k(\R) = {G_{k}} \rtimes \R^k;$$ the latter is} the group  of volume-preserving} affine transformations of $\R^k$. Also  put $${{\Gamma}_{k}} := \SL_k(\Z)    \text{ and } {\widehat{\Gamma}_k} := \ASL_k(\Z) ={{\Gamma}_{k}} \rtimes \Z^k.$$  {Elements of ${\widehat{G}_{k}}$ will be denoted by $
 \langle g,\vw\rangle$ where $g\in {G_{k}}$ and $\vw\in\R^k$; that is, {$\langle g,\vw\rangle$} is the affine transformation $\x\mapsto g\x + \vw$.} Denote by ${\widehat{X}_k}$ the space of {translates of unimodular lattices in $\R^k$; elements of ${\widehat{X}_k}$ will be referred to as {\sl {unimodular grids}}.
Clearly ${\widehat{X}_k}$ is canonically identified with  ${\widehat{G}_{k}}/{\widehat{\Gamma}_k}$  via $$\langle g,\vw\rangle{\widehat{\Gamma}_k}   \in{\widehat{G}_{k}}/{\widehat{\Gamma}_k}\quad\longleftrightarrow\quad g\Z^k +\vw\in {\widehat{X}_k}.$$ 
{Similarly, ${X}_k := {G_{k}}/{{\Gamma}_{k}}$ is {identified with} the space of unimodular lattices in $\R^k$ {(i.e.\ unimodular grids containing the zero vector}).}  Note that ${\widehat{\Gamma}_k}$ (resp.\ ${\Gamma}_k$) is a lattice in ${\widehat{G}_{k}}$ (resp.\ ${G_{k}}$).  {We will denote by $\widehat\mu$ ({resp.\ $\mu$}) the normalized Haar measures on ${\widehat{X}_k}$ and ${X_k}$} {respectively.}

\smallskip
{Now fix $m,n\in\N$ with $m+n=k$},  and {for} $t \in\R$ let \eq{diagflow}{{{g_t}}
:= \diag(e^{{{t}}/m},\dots,e^{{{t}}/m},e^{-{{t}}/n},\dots,e^{-{{t}}/n}),}
where there are $m$ copies of $e^{{{t}}/m}$ and $n$ copies of $e^{-{{t}}/n}$. 
{The so-called {\sl expanding horospherical subgroup} of ${\widehat{G}_{k}}$ with respect to $\{g_t: t > 0\}$ is given by 
\eq{h}{H:= \{u_{A, \mathbf{b}}: A \in M_{m,n}, \mathbf{b} \in \mathbb{R}^m\}, \text{ where }u_{A, \mathbf{b}}:= \left\langle\left( \begin{array}{cc} I_m & A \\
0 & I_n \\ \end{array} \right) , \left( \begin{array}{c} \mathbf{b} \\ 0 \end{array} \right)\right\rangle.}
On the other hand,
\eq{htilde}{\tilde H :=\left\{ \left\langle \left( \begin{array}{cc} P & 0 \\ R & Q \end{array} \right), \left( \begin{array}{c} 0 \\ \mathbf{d} \end{array} \right) \right\rangle\left|\begin{aligned}P\in M_{m,m},\ &Q\in M_{n,n},\ \det(P) \det(Q)=1\\ &R\in M_{n,m},\   \mathbf{d} \in \mathbb{R}^n\end{aligned}\right.\right\}}
is a subgroup of ${\widehat{G}_{k}}$ complementary to $H$ which is {\sl non-expanding} with respect to conjugation by $g_t$, $t\ge 0$: it is easy to see that
\eq{conj}{{{g_t}}\left\langle \left( \begin{array}{cc} P & 0 \\ R & Q \end{array} \right), \left( \begin{array}{c} 0 \\ \mathbf{d} \end{array} \right) \right\rangle g_{-t} = \left\langle \left( \begin{array}{cc} P & 0 \\ e^{-\frac{m+n}{mn}t}R & Q \end{array} \right), \left( \begin{array}{c} 0 \\ e^{-t/n}\mathbf{d} \end{array} \right) \right\rangle.
}
Let us also denote
\eq{lambdaab} {\Lambda_{A, \mathbf{b}}:=u_{A, \mathbf{b}} \mathbb{Z}^{k} {= \left\{\left( \begin{array}{c} A\vq + \mathbf{b} - \vp \\ \vq \end{array} \right): \vp\in\Z^m,\ \vq\in\Z^n\right\}}.}
The reduction of  \di\ properties of $(A,\mathbf{b})$ to the behavior of the $g_t$-trajectory of $\Lambda_{A, \mathbf{b}}$ described {below} mimics the classical Dani correspondence for homogeneous \da\  \cite{dani, KM} and dates back to    \cite{K99} {(see also more recent papers  \cite{Shapira-Cassels, ET, GV})}. 
The crucial role is played by a function {$\Delta: {\widehat{X}_k}\to [-\infty,\infty)$} given by 
\eq{defdelta}{\Delta (\Lambda):= 
\log\inf_{\vv\in \Lambda}\|\vv\|.}
{Note that $\Delta(\Lambda) = -\infty$ if and only if $\Lambda \ni 0$. Also 
it is easy to see that $\Delta$ is uniformly continuous outside of the set where it takes small values:
\smallskip

\begin{lem}\label{uc}{For any $z\in\R$, $\Delta$ is uniformly continuous on the set $ \Delta^{-1}\big([z,\infty)\big)$. That is,  for any $z\in \R$ and any $\vre>0$ there exists a neighborhood $U$ of {the} identity in ${\widehat{G}_{k}}$ such that   whenever $\Delta(\Lambda) \ge z$ and ${g}\in U$, one has $|\Delta(\Lambda)  - \Delta({g}\Lambda) | < \vre$.}\end{lem}
\begin{proof} 
Let $c>1$, $z \in \R$.   Choose $\delta>0$ so that $$c^{-1}\|\vv\| \leq \|\vv+\vw\| \leq c \|\vv\|$$ whenever $\|\vw\| \leq \delta$ and $\log \|\vv\| \geq z - \log c$.  Then if $\log  \|\vv\| \geq z$, $\|\vw\| < \delta$ and the operator norms of both $g$ and $g^{-1}$ are not greater than $c$ {(the latter two conditions define an open neighborhood $U$ of {the} identity in $\widehat{G}$ such that $\langle g,\vw\rangle\in U$)},  we have
$$\frac{\|\vv\|}{c^2} \leq \frac{\|g\vv\|}{c}  \leq \|g\vv + \vw\| \leq  c\cdot \|g\vv\|  \leq c^2 \cdot \|\vv\|.$$ 
Thus if $\Delta(\Lambda)\geq z$ {and $\langle g,\vw\rangle\in U$}, we have $$\Delta(\Lambda)-2\log c \leq \Delta(g\Lambda +\vw)\leq \Delta(\Lambda)+ 2\log c.$$   Since $c>1$ is arbitrary, $\Delta$ is uniformly continuous on $\Delta^{-1}\big([z,\infty)\big)$.
\end{proof}

{Another important feature of $\Delta$ is that it is 
unbounded from above;} indeed, the {grid} $$\diag(1,\dots,1,\tfrac14e^{-z}, 4e^z)  \Z^k+ (0,\dots,0,2e^z)$$  is disjoint from the ball centered at $0$ of radius $e^z$. Consequently,    sets ${ \Delta^{-1}\big([z,\infty)\big)} $  have non-empty interior for all $z\in\R$.

\smallskip
{Let us now describe a basic special case  of the correspondence between inhomogeneous improvement of Dirichlet's Theorem and dynamics {on ${\widehat{X}_k}$. The next lemma} is essentially an inhomogeneous analogue of \cite[Proposition 2.1]{KWe}:
\begin{lem}\label{corr_c/t}Let $C>0$ and put $z = \frac {\log C}{m+n}$. Then $(A,\mathbf{b})\in \admn(C\psi_1)$ if and only if $\Delta({{g_t}} \Lambda_{A, \mathbf{b}}) < z$ 
for 
all large enough ${t > 0}$.
\end{lem}
\begin{proof} 
For $T > 1$, put $\psi({T})=C\psi_1(T)= C/{T}$, and 
define 
$${t} := \log T - \frac n{m+n} \log C\quad\Longleftrightarrow\quad T = C^{\frac m{m+n} }e^t.$$
Then $\psi(T) =  C^{\frac n{m+n} }e^{-t}$, and the system \equ{ideq} can be written as
$${\|A \vq+\mathbf{b}-\vp\|^m< C^{\frac n{m+n} }e^{-t}\hspace{10mm}\|\vq\|^n< C^{\frac m{m+n} }e^t},$$
which is 
 the same a $$e^{ {t}/m} \|A\vq +\mathbf{b}-\vp \| < C^{\frac 1{m+n}} \hspace{10mm} e^{-{t}/n} \| \vq \| <  C^{\frac 1{m+n}}.$$ 
In view of \equ{diagflow}, \equ{lambdaab} and \equ{defdelta}, the solvability of \equ{ideq} in $(\vp,\vq)\in\Z^{m+n}$ is equivalent to $$\Delta({{g_t}} \Lambda_{A, \mathbf{b}})< \frac {\log C}{m+n} = z,$$
and the conclusion follows.
\end{proof}}

{We are going to use the above lemma and  the ergodicity of the ${{g_t}}$-action on ${\widehat{X}_k}$ to compute {the} Lebesgue} measure of $\admn(C\psi_1)$. The proof contains a Fubini Theorem argument (following \cite[Theorem 8.7]{KM} and dating back to \cite{dani}) used to pass from an almost-everywhere statement for lattices to an almost-everywhere statement for pairs in $M_{m,n}\times \R^m$. We will refer to this argument twice more in the sequel.

\begin{prop}\label{c/t}{{For any $m,n\in \N$ and} any $C>0$, the set $\admn(C\psi_1)$ has Lebesgue measure zero.}\end{prop}

\begin{proof} 
%
%
Suppose {${U}$ is a subset of} $M_{m,n}\times \mathbb{R}^m${
($\cong H$  as in \equ{h})} {of positive Lebesgue measure}  such that $ \Delta( {{g_t}} \Lambda_{A, \mathbf{b}})< \frac{\log C}{{m+n}}$ for any $(A, \mathbf{b})\in {U}$ and all large enough ${t}$.
Then 
{there exists a neighborhood $V$ of  identity in $\tilde H$ {
as in \equ{htilde}} such that for all $  {g}\in V$,  $(A, \mathbf{b})\in  {U}$ and all large enough ${t}$, 
\eq{less2}{\Delta( {{ g_t}  {g}}  \Lambda_{A, \mathbf{b}}) = \Delta( {{ g_t}  {g}}g_t^{-1} g_t  \Lambda_{A, \mathbf{b}})< \frac{\log C}{{m+n}}{ +1}.} Indeed, one can use Lemma \ref{uc}   and \equ{conj} to 
{choose $V$} such that 
if \equ{less2} does not hold for ${{g}\in V}$, then $\big|\Delta( {{ g_t}  {g}}g_t^{-1} g_t  \Lambda_{A, \mathbf{b}}) -
\Delta( {{ g_t}}  \Lambda_{A, \mathbf{b}})\big| < 1$.} But since the product map $\tilde H \times H\to {\widehat{G}_{k}}$ is a 
local diffeomorphism, $V\times {U}$ { is mapped   onto a set of positive measure.} It follows that  $\Delta( {{g_t}}  \Lambda)< {\frac{\log C}{{m+n}} + 1}$ for all large enough $t$ {and} {for a} set of lattices $\Lambda$ of positive {Haar} measure {in ${\widehat{X}_{k}}$}. 

On the other hand from  Moore's ergodicity theorem \cite{Mo} together with the ergodicity criterion of Brezin and Moore { (see \cite[Theorem 6.1]{BM} or \cite[Theorem 6]{Ma})}
it follows that every {unbounded} subgroup of $G_k$ -- in particular, $\{{{g_t}} : {t}\in\mathbb{R} \}$ as above -- acts ergodically on  
${\widehat{X}_k}$. 
Since for any ${C> 0}$ the   set ${ \Delta^{-1}\big([{\frac{\log C}{{m+n}} + 1},\infty)\big)} $  has a non-empty interior, it follows that {$\mu$-}almost every  $\Lambda\in {\widehat{X}_k}$ {must visit any such set at unbounded times under the action of $g_t$}, a contradiction.
\end{proof}
\section{A correspondence between Dirichlet improvability and dynamics}
} \label{corr}

 \noindent {{Lemma \ref{corr_c/t}} relates }{the complement of $\admn(C\psi_0)$ to the set of {grids}} visiting 
{certain} ``target" subsets of ${\widehat{X}_k}$
 at unbounded times under the diagonal flow ${{g_t}}$.  This is the special case where the target does not change with the time parameter ${t}$.  For general non-increasing $\psi$, we get a family of ``shrinking targets" $\Delta^{-1}\big([{z}_\psi(t),\infty)\big)$ (which in fact are shrinking only in a weak sense --- see Remark \ref{quasi1}), where ${z}_\psi$ is gotten by the following change of variables, known as the Dani Correspondence:

\smallskip
\begin{lem}[See Lemma 8.3 in \cite{KM}]\label{dani1} Let positive integers $m,n$  and ${T}_0\in \R_+$ be given. Suppose $\psi : [{T}_0, \infty)\rightarrow \R_+$ is a continuous, non-increasing function. Then there exists a unique continuous function $$z={z}_\psi
: [{t_0},\infty)\rightarrow \mathbb{R},$$ where ${t_0}:= \frac{m}{m+n}\log {T}_0-\frac{n}{m+n}\log \psi({T}_0)$, such that
\begin{enumerate}
\item the function $ {t}\mapsto {t}+n{z}({t})$ is strictly increasing and unbounded;
\item the function ${t}\mapsto {t}-m{z}({t})$ is nondecreasing;
\item $\psi(e^{{t}+n{z}({t})})=e^{-{t}+m{z}({t})}$ for all ${t}\geq {t_0}$. 
\end{enumerate}
\end{lem}
 
{
 \begin{remark} \rm The function ${z}$ of Lemma \ref{dani1} differs from {the function $r$} of \cite[Lemma 8.3]{KM} by a minus sign. This reflects the difference between the asymptotic and uniform approximation problems. \end{remark}}

\begin{remark}\label{quasi1}\rm For future reference, we point out that properties {\rm (1)} and {\rm (2)} of Lemma \ref{dani1} imply that any ${z = {z}_\psi}$ 
{does not oscillate too wildly. Namely,}
$${z}({{s}})-
{\tfrac1m}
\leq {z}(u)\leq {z}({{s}})+
{\tfrac1n}
 ~~~~ \textit{whenever} ~~~ s\leq u\leq s+1.$$ \end{remark}




Now we can state a general version of the correspondence between {the} improvability of {the inhomogeneous Dirichlet theorem} and dynamics on ${\widehat{X}_k}$, generalizing the first paragraph of the proof of Theorem \ref{c/t}.

\begin{lem}\label{inhomogeneousshrinkingtarget} \textit{Let $\psi: [{T}_0,\infty)\rightarrow \R_+$ be a non-increasing  continuous function, and let {$z={z}_\psi$} be the function associated to $\psi$ by Lemma \ref{dani1}. The pair $(A, \mathbf{b})$ is in $\admn(\psi)$ if and only if $\Delta({{g_t}}\Lambda_{A, {b}})<{z}_{\psi}({t})$ for all sufficiently large ${t}$. } \end{lem}

\begin{proof} {We argue as in the proof 
of Lemma \ref{corr_c/t}.} Since ${t} \mapsto{t}+n{z}({t})$ is increasing and unbounded, $(A, \mathbf{b}) \in \admn(\psi)$ if and only if for all large enough ${t}$ we have
$$\|A\vq+\mathbf{b}-\vp\|^m< \psi (e^{{t}+n{z}({t})})=e^{-{t}+m{z}({t})}, \hspace{10mm} \|\vq\|^n < e^{{t}+n{z}({t})}$$
for some $\vq\in \mathbb{Z}^n$, $\vp\in \mathbb{Z}^m$. This is the same as the solvability of
$$e^{{t}/m} \|A\vq+\mathbf{b}-\vp \| < e^{{z}({t})},\hspace{10mm} e^{-{t}/n} \|\vq\| < e^{{z}({t})}, $$ \noindent which is the same as $\Delta({{g_t}} \Lambda_{A,\mathbf{b}})< {z}_\psi({t})$.\end{proof}

Thus a pair fails to be $\psi$-Dirichlet if and only if {the} associated {grid} visits the ``target" $\Delta^{-1}\big([{z}_\psi({t}),\infty)\big)$ at unbounded times ${t}$ under the flow ${{g_t}}$.   This is known as a ``shrinking target  phenomenon." {Our next goal is to recast condition \equ{psisum} using the function ${z}_\psi$:
\begin{lem}\label{dani} Let $\psi : [{T}_0, \infty) \rightarrow \R_+, {T}_0\geq 0$ be a non-increasing continuous function, and {$z=
z_\psi$} the function associated to $\psi$ by Lemma \ref{dani1}. Then we have
$$\sum_{j=\lceil{T}_0\rceil}^\infty \frac{1}{{j}^2\psi({j})}  < \infty ~~ \textit{if and only if} ~~~ \sum_{{t=\lceil t_0 \rceil}}^\infty e^{-(m+n){z}({t})}  < \infty .$$
\end{lem}
\begin{proof} We follow the lines of the proof of  \cite[Lemma 8.3]{KM}.
Using the monotonicity of $\psi$ and Remark \ref{quasi1}, we may replace the sums with integrals $$\int_{{T}_0}^\infty {x}^{-2} \psi({x})^{-1}\,d{x}\quad\text{  and }\quad\int_{{t_0}}^\infty e^{-(m+n){z}({t})}\,d{t}$$ respectively.  Define $$P:= -\log\circ \psi\circ \exp: [{T}_0,\infty) \rightarrow \R\quad\text{  and }\quad\lambda(t):= {t}+n{z}({t}).$$ Since $\psi(e^\lambda) =e^{-P(\lambda)}$, we have
$$\int_{{T}_0}^\infty {x}^{-2} \psi({x})^{-1}\,d{x}=\int_{\log {T}_0 }^\infty \psi(e^\lambda)^{-1}e^{-\lambda}d\lambda= \int_{\log {T}_0}^\infty e^{P(\lambda)-\lambda}\,d\lambda.$$ Using $P\big(\lambda({t})\big)= {t}-m{z}({t})$, we also have \begin{equation*}\begin{split}
\int_{{t_0}}^\infty e^{-(m+n){z}({t})}\,d{t} &= \int_{\log{T}_0}^\infty e^{-(m+n){z}\left(\frac{m\lambda}{m+n}+\frac{nP(\lambda)}{m+n}\right)}\,d\left[\frac{m}{m+n}\lambda +\frac{n}{m+n}P(\lambda)\right]\\
&=\frac{m}{m+n}\int_{\log{T}_0}^\infty e^{P(\lambda)-\lambda} \,d\lambda +\frac{n}{m+n}\int_{\log{T}_0}^\infty e^{-\lambda} e^{P(\lambda)}\,dP(\lambda)\\
=\frac{m}{m+n}\int_{\log{T}_0}^\infty e^{P(\lambda)-\lambda} \,d\lambda &+\frac{n}{m+n}\int_{\log{T}_0}^\infty e^{P(\lambda)-\lambda} \,d\lambda + \frac{n}{m+n}\left(\lim_{\lambda\rightarrow \infty} e^{P(\lambda)-\lambda}-1\right),\end{split}\end{equation*}
 where we integrated by parts in the last line.  Since all these quantities (aside from the constant $-1$) are positive, the convergence of $\int_{{t_0}}^\infty e^{-(m+n){z}({t})} \,d{t}$ implies the convergence of $\int_{\log{T}_0}^\infty e^{P(\lambda)-\lambda}\,d\lambda$.  Conversely, suppose $\int_{\log{T}_0}^\infty e^{P(\lambda)-\lambda}\,d\lambda$ converges, yet $\int_{{t_0}}^\infty e^{-(m+n){z}({t})}\,d{t}$ diverges.  Then since $u\mapsto \int_{{t_0}}^u e^{(m+n){z}({t})}\,d{t}$ is increasing in $u$, and the first two terms of the sum above converge, we must have $e^{P(\lambda)-\lambda}$ eventually increasing in $\lambda$ (recall that $\lambda$ is an increasing and unbounded function). But this contradicts the convergence of $\int_{\log({T}_0)}^\infty e^{P(\lambda)-\lambda}\,d\lambda$.
\end{proof}}

{Now we are ready to reduce Theorem \ref{mainthm} to the following statement concerning dynamics on ${\widehat{X}_k}$:}

{\begin{thm}\label{dynamicalmainthm} Fix $k\in \N$ and let $\{g_t: t\in\R\}$ be a diagonalizable {unbounded} one-parameter subgroup of ${{G_{k}}}$. Also take an arbitrary sequence $\{{z}(t): t\in \N\}$ of real numbers. Then the set
\eq{limsupset}{\big\{\Lambda\in {\widehat{X}_k}: \Delta({g_t}\Lambda)\ge {z}(t)\text{ for infinitely many }t\in\N\big\}}
is null (resp.\ conull) if the sum
\eq{deltasum}{
\sum_{t=1}^\infty e^{-k{z}(t)}
}
converges (resp.\ diverges).
\end{thm}}


\begin{proof}[Proof of Theorem \ref{mainthm} assuming Theorem \ref{dynamicalmainthm}]
{Suppose that the series \equ{psisum} converges, and take ${z}(t) = {z}_\psi(t)$, the function associated to $\psi$ by Lemma \ref{dani1}.  In view of Lemma \ref{dani}, the series  \equ{deltasum} converges as well. In particular, it follows that ${z}(t) \ge 0$ for {all} large enough $t\in \N$, and also that 
$\sum_{t=1}^\infty e^{-k\left({z}(t) - C\right)}< \infty$
for any $C > 0$. Take $g_t$ as in \equ{diagflow}; Theorem \ref{dynamicalmainthm} then implies that  
\eq{measure0}{
{\widehat{\mu}}\left(\big\{\Lambda\in {\widehat{X}_k}: \Delta({g_t}\Lambda)\ge {z}(t) - C\text{ for infinitely many }t\in\N\big\}\right) = 0.
}
Suppose   that the Lebesgue measure of $\admn(\psi)^c$  is positive.
Lemma \ref{inhomogeneousshrinkingtarget}  asserts that  there exists  a set $U$ of positive measure consisting of pairs $(A, {\vb})$ for which $\Delta({{g_t}}\Lambda_{A, {\vb}})\ge {z}({t})$ for an unbounded  set of ${t} \ge 0$. {Then, using $z(t)\geq 0$ and Lemma \ref{uc},} we can replace $t$ with its integer part: 
$$\Delta({{g_{\lfloor t\rfloor}}\Lambda_{A, {\vb}}}) = \Delta({g_{\left(\lfloor t\rfloor - t\right)}}g_t\Lambda_{A, {\vb}})\ge \Delta({{g_t}}\Lambda_{A, {\vb}}) - c \ge {z}({t})- c \ge {z}({\lfloor t\rfloor})-c - 1/m,$$
{where $c$ is a positive constant and} the last inequality follows from Remark  \ref{quasi1}. Therefore we get
$\Delta({{g_t}}\Lambda_{A, {\vb}})\ge {z}({t})- {c-\frac1m}$ for an unbounded  set of ${t} \in \N$ as long as $(A, {\vb})\in U$.}

{Now recall the groups $H$ and $\tilde H$ from equations \equ{h} and \equ{htilde}.  As in the proof of Proposition \ref{c/t}, we may identify $U$ with a subset of $H$ and, using the uniform continuity of $\Delta$ (Lemma~\ref{uc}){,} find a neighborhood of identity $V \subset \tilde H$ such that for all ${g}\in V$ and  $(A, \mathbf{b})\in U$ 
$$\Delta( {{ g_t}{g}}  \Lambda_{A, \mathbf{b}}) = \Delta( {{ g_t}{g}}g_t^{-1} g_t  \Lambda_{A, \mathbf{b}})\ge  \Delta( {{ g_t}}  \Lambda_{A, \mathbf{b}})-1$$ for all  $t\ge 0$, hence $\Delta({{g_t}}{g}\Lambda_{A, {\vb}})\ge {z}({t})- 1-c-\frac1m$ for an unbounded  set of ${t} \in \N$. Since the product map $\tilde H \times H\to {\widehat{G}_{k}}$ is a local diffeomorphism, the image of $V\times U$ is a set of positive measure in $G_k$, contradicting \equ{measure0}.}

{The proof of the divergence case proceeds along the same lines. 
If \equ{psisum} diverges, by Lemma \ref{dani} so does  \equ{deltasum}. Define ${z'}(t) := \max\big({z}(t),0\big)$, then we have
$\sum_{t=1}^\infty e^{-k\left({z'}(t) \right)}= \infty$ as well, therefore
$\sum_{t=1}^\infty e^{-k\left({z'}(t) + C\right)}= \infty$
for any $C > 0$. In view of
 Theorem \ref{dynamicalmainthm},  
\eq{measure1}{\text{the set }\big\{\Lambda\in {\widehat{X}_k}: \Delta({g_t}\Lambda)\ge {z'}(t) + C\text{ for infinitely many }t\in\N\big\}\text{ has full measure.}}
Now assume  that the set $\admn(\psi)$  has positive measure. Then one can, using Lemma \ref{inhomogeneousshrinkingtarget}, choose  a set $U$ of positive measure consisting of pairs $(A, {\vb})$ for which $$\Delta({{g_t}}\Lambda_{A, {\vb}})<  {z}({t})\le {z'}(t)$$ for all large enough ${t}$. Then, as before, using Lemma \ref{uc} with $z=0$  and \equ{conj}, one finds a neighborhood of identity $V \subset \tilde H$ such that for all ${g}\in V$ and  $(A, \mathbf{b})\in U$,
$$\Delta( {{ g_t}{g}}  \Lambda_{A, \mathbf{b}}) = \Delta( {{ g_t}{g}}g_t^{-1} g_t  \Lambda_{A, \mathbf{b}}) < \max\big( \Delta( {{ g_t}}  \Lambda_{A, \mathbf{b}}),0\ )+1$$ for all  $t\ge 0$; hence $\Delta({{g_t}}{g}\Lambda_{A, {\vb}}) < {z'}({t})+ 1$ for all large enough $t$. Again using the local product structure of  $ {\widehat{G}_{k}}$, one concludes that  the image of $V\times U$ in $ {\widehat{X}_k}$ is a set of positive measure, contradicting \equ{measure1}.}\end{proof}
\medskip

\ignore{It is clear  that the convergence part {of the above theorem} is immediate from the classical Borel-Cantelli lemma, and the only challenge is to prove the divergence part. In the next section we will deduce the above theorem from one of the main results of \cite{KM, KMapp}.
%
{In {addition to using}  the correspondence of Lemma \ref{inhomogeneousshrinkingtarget}, 
{for the reduction of Theorem \ref{dynamicalmainthm} to Theorem \ref{mainthm}} it remains to:
\begin{itemize}
\item pass from almost all $\Lambda\in {\widehat{X}_k}$ to almost all $\Lambda$ of the form $\Lambda_{A,\mathbf b}$, similarly to what is done in the proof of Proposition \ref{c/t};
\item pass from arbitrary $t > 0$ to $t\in \N$ (this can be achieved using Remark \ref{quasi1});
\item and, finally, show that the series \equ{deltasum} converges/diverges if and only if so does \equ{psisum}. 
\end{itemize}
This will all be accomplished in \S\ref{mainproof}.
}}

{We are now left with the task of proving Theorem \ref{dynamicalmainthm}. The proof will have two ingredients. In the next section we will establish a dynamical Borel-Cantelli Lemma (Theorem \ref{thm4.3}) showing that the limsup set
\equ{limsupset}
is null or conull according to the convergence or divergence of the  series
\eq{pdsum}{
\sum_{t=1}^\infty 
{\widehat{\mu}}\big(\{\Lambda \in {\widehat{X}}_k: \Delta(\Lambda) \ge  {z}(t)\}\big).
}
The proof is based on the methods of \cite{KM, KMapp}; namely, it uses the exponential mixing of the $g_t$-action on ${\widehat{X}}_k$, as well as {the} so-called DL property of $\Delta$. The latter will be established in \S \ref{deltadl}. Moreover, there  we will relate \equ{deltasum} and \equ{pdsum} by showing that the 
summands in \equ{pdsum} are {equal to $e^{-k{z}(t)}$ up to a constant} (Theorem \ref{thmDL}).}

\section{{A general dynamical Borel-Cantelli lemma and exponential mixing}} \label{section}
\noindent In this section we let $G$ be a 
Lie
group  
and  $\Gamma$ 
a lattice in $G$. Denote by $X$ the \hs\ $\ggm$ and by $\mu$ the $G$-invariant probability measure on {$X$.}
In what follows, $\|\cdot\|_p$ will stand for the $L^p$-norm. Fix a basis $\{Y_1,\dots,Y_n\}$ 
for the Lie algebra
$\frak g$
of
$G$, and, given a smooth
 function $h \in C^\infty(X)$ and $\ell\in\Z_+$, define the ``{\sl $L^2$, order $\ell$" Sobolev norm} $\|h \|_{2,\ell}$ of $h $ by 
 $$
 \|h \|_{2,\ell} \df \sum_{|\alpha| \le \ell}\|D^\alpha h \|_2,
 $$
 where 
 $\alpha = (\alpha_1,\dots,\alpha_n)$ is a multiindex, $|\alpha| = \sum_{i=1}^n\alpha_i$, and $D^\alpha$ is a differential operator of order $|\alpha|$ which is a monomial in  $Y_1,\dots, Y_n$, namely $D^\alpha = Y_1^{\alpha_1}\cdots Y_n^{\alpha_n}$.
This definition depends on the basis,
{however} a change of basis
would only  distort
$ \|h \|_{2,\ell}$
by a bounded factor. We also let 
$$C^\infty_2(X) = \{h \in C^\infty(X): \|h \|_{2,\ell} < \infty\text{ for any }\ell \in \Z_+\}.$$
Fix  a right-invariant
Riemannian metric on $G$ and the corresponding metric `dist' on $X$. For $g\in G$, let us denote by $\|g\|$ the
distance between $g\in G$ and the identity element of $G$. Note that $\|g\| = \|g^{-1}\|$ due to the right-invariance of the metric.
{\begin{defn}\label{expmix} Let $L$ be a subgroup of $G$. Say  that the $L$-action on  $X$ is {\sl exponentially
mixing} 
if there exist   $\gamma ,E > 0$ and $\ell\in\Z_+$
such that for any   $\varphi, \psi \in C^\infty_2(X)$  and for any $g\in {L}$ one has
\begin{equation*}\label{em}\tag{EM}
\left| \langle g\varphi ,\psi \rangle - \int_X\varphi\,d\mu \int_X\psi\,d\mu  \right| \le E{e^{ - \gamma \|g\|}\left\| \varphi  \right\|_{2,\ell}} {\left\| \psi  \right\|_{2,\ell}}\,.\end{equation*}
Here $\langle \cdot ,\cdot \rangle$ stands for the inner product in $L^2(X,\mu)$.\end{defn}}

We also need two more definitions from \cite{KM, KMapp}. 
{\begin{defn}\label{expdiv}A sequence of elements $\{f_t: t\in\N\}$ of elements of $G$ is called  {\sl exponentially divergent} if 
\begin{equation}\label{ed}
\sup_{t\in\N}\sum_{s=1}^\infty e^{-\gamma\|f_sf_t^{-1}\|} < \infty \quad
\forall\,\gamma > 0.
\end{equation}\end{defn}}

 Now let $\Delta$ be a {real-valued} function on $X$, and for $z \in\R$ denote 
$$\pd(z) \df\mu\left(\Delta^{-1}\big([z,\infty)\big)\right).$$ 
{\begin{defn}\label{defdl}Say that $\Delta$ is {\sl DL\/} (an abbreviation for
``distance-like'') if {there exists $z_0 \in\R$ such that {$\pd(z_0) > 0$ and}
\begin{itemize}
\item[\rm (a)] $\Delta$ is  uniformly continuous on $\Delta^{-1}\big([z_0,\infty)\big)$; that is,    $\forall\,\vre>0$ there exists a neighborhood $U$ of  identity in ${G}$ such that   for any $x\in X$ with $\Delta(x) \ge z_0$,
$$  g\in U \quad\Longrightarrow \quad|\Delta(x)  - \Delta(gx) | < \vre;$$
\item[\rm (b)] the function $\pd$ 
does not decrease very fast; more precisely,  
\begin{equation}\label{dl}
\exists\, c,\delta > 0
\text{ such that } 
{\pd(z ) \ge c  \pd(z - \delta)}
\ \ \ \forall\, z\ge {z_0}.
\end{equation}
\end{itemize}}\end{defn}}
\ignore{Now let $\Delta$ be a real-valued nonnegative function on $X$, and for $z\ge 0$ denote 
$$\pd(z) \df\mu\left(\Delta^{-1}\big([z,\infty)\big)\right).$$ 
Say that $\Delta$ is {\sl DL\/} (an abbreviation for
``distance-like'') if it is  uniformly continuous {\color{red} on super-level sets $\Delta^{-1}[z,\infty)$} (in particular, bounded on compact subsets of $X$), and the function $\Phi_\Delta$ 
does not decrease very fast, more precisely, if
\begin{equation}\label{dl}\tag{DL}
\exists\, c, \epsilon > 0
\text{ such that } 
\Phi_\Delta(z + \epsilon) \ge c  \Phi_\Delta(z)
\ \forall\, z\in\R.
\end{equation}}

The next theorem is a 
{direct consequence} of \cite[Theorem 1.3]{KMapp}:

\begin{thm}\label{thm4.3} 
Suppose that 
the {action of a subgroup $L \subset G$} on $X$ is exponentially mixing.
Let $\{f_{t} : {t}\in \N\}$ be a sequence of elements of {$L$} satisfying \eqref{ed},    and let  $\Delta$ be a DL function on $X$. Also let $\{{z}({t}): t\in\N\}$ be
a sequence of 
{real} numbers. 
Then {the set
\eq{limsupsetgeneral}{\big\{\Lambda\in {X}: \Delta({g_t}\Lambda)\ge {z}(t)\text{ for infinitely many }t\in\N\big\}}
is null (resp.\ conull) if the sum
\eq{div}{
{ \sum_{{t}=1}^{\infty} \Phi_\Delta\big(z(t)\big)}}
converges (resp.\ diverges).}
\end{thm}

{\begin{proof} The convergence case is immediate from the classical Borel-Cantelli Lemma. The divergence case is established in \cite{KM, KMapp} for $L = G$, but the argument applies verbatim if $G$ is replaced by {a} subgroup.\end{proof}}

From now on we are going to take  $k \ge 2$ and consider {the case} $G =  {\widehat{G}}_k$, $\Gamma = {\widehat{\Gamma}_{k}}$, $X = {\widehat{X}}_k$, and $L = {G_{k}}$, with notation as in the previous section. Then we have}

\begin{thm}\label{thmEM} 
{The ${G_{k}}$-action} on ${\widehat{X}_k} = {\widehat{G}_{k}}/{\widehat{\Gamma}_k}$  
is exponentially mixing.
\end{thm}

\begin{proof} {According to} \cite[Theorem 3.4]{KM},
{exponential mixing holds whenever} 
the regular representation of {${G_{k}}$} on the space $L_0^2({\widehat{X}_k})$ (functions in $L^2({\widehat{X}_k})$ with integral zero) is isolated in the Fell topology from the trivial representation. 
This is immediate if  $k > 2$ since in this case ${G_{k}}$ has Property (T). 

If  $k=2$, let us write $L_0^2( {\widehat{X}}_2)$ as a direct sum of two spaces: functions invariant under the action of $\R^2$ by translations, and its orthogonal complement. The first representation is isomorphic to the regular representation of {$\SL_2(\R)$} on {$L_0^2\big(\SL_2(\R)/\SL_2(\Z)\big)$}, 
which is isolated from the trivial representation by \cite[Theorem 1.12]{KM}. As for the second component, one can use \cite[Theorem V.3.3.1]{HT} (see also  \cite[Theorem 4.3]{GGN}) which asserts that for any unitary representation $(\rho,V)$ of {$\ASL_2(\R)$} with no nonzero vectors fixed by $\R^2$, the restriction of $\rho$ to {$\SL_2(\R)$} is {\sl tempered}, that is, there exists a dense set of vectors in $V$ whose matrix coefficients are in $L^{2+\vre}$ for any $\vre > 0$. Exponential mixing thus follows from \cite[Theorem 3.1]{KS}, which establishes exponential decay of matrix coefficients of strongly $L^p$ irreducible unitary representations of connected semisimple centerfree Lie groups.  See also a preprint  \cite{edwards} for more precise estimates. \end{proof}


{Now let $\Delta$ be the function  on ${\widehat{X}_k}$ defined by \equ{defdelta}.  
In the next section we are going to establish the following two-sided estimate for the measure of super-level sets of $\Delta$:
\begin{thm}\label{thmDL} For any $k\ge 2$ there exist $c, C>0$ such that 
\eq{ddl}{ce^{-kz}\leq \Phi_\Delta(z) \leq Ce^{-kz} 
\text{ for all }z\geq 0. }
\end{thm}
This is all one needs to settle 
Theorem \ref{dynamicalmainthm}:}

\ignore{\begin{itemize}
\item observe that for a one-parameter diagonalizable $\{g_t : t\in\R\}\subset G_k$ one has $\|g_t\| \ge {\alpha t}$ for some $\alpha > 0$, which immediately implies \eqref{ed};
\item use Theorem \ref{thmEM}, that is, the exponential mixing of the $G_k$-action on ${\widehat{X}_k}$;
\item {\color{red}\sout{show that to prove Theorem \ref{dynamicalmainthm} it suffices to consider the case when $r(t) \ge 0$ for large enough $t$, and thus one can replace $\Delta$ as in \equ{defdelta} with $\Delta'\df \max(\Delta,0)$}};
\item show that the function {\color{red}$\Delta$}  is DL; this will be done in the next section. 
\end{itemize}}

\begin{proof}[{Proof of Theorem \ref{dynamicalmainthm} modulo  Theorem  \ref{thmDL}}]
Let 
$\{g_t: t\in\R\}$ be a diagonalizable {unbounded}  one-parameter subgroup of ${{G_{k}}}$.  By Theorem \ref{thmEM}, the action of ${G_{k}}$ on ${\widehat{X}_k}$ is exponentially mixing. Observe also that one has $\|g_t\| \ge {\alpha t}$ for some $\alpha > 0$, which immediately implies \eqref{ed}. {It is 
easy to see that  \equ{ddl} implies \eqref{dl}  with $z_0 = 0$, and part (a) of Definition \ref{defdl} is given by   Lemma \ref{uc}. The conditions of Theorem \ref{thm4.3} are therefore met, and Theorem \ref{dynamicalmainthm} follows.}
\end{proof}


\section{$\Delta$ is  Distance-Like: {a warm-up}}
\noindent {For the rest of the paper we keep the notation $${{G}} = {G_{k},} \ {\widehat{G}} = {\widehat{G}_k} = {G_{k}}\rtimes \R^k, 
\ X = X_k = G_k/\Gamma_k, \ \widehat{X} = \widehat{X}_k = \widehat{G}_k/\widehat{\Gamma}_k,
$$ and 
let $\mu$ (resp.\ $\widehat\mu$) be the Haar probability measure on $X$ (resp.\ $\widehat{X}$).}
We denote by $\mu_G$ and $\mu_{{\widehat{G}}}$ the left-invariant Haar measures on $G$ and ${\widehat{G}}$ respectively which are locally pushed forward to $\mu$ and ${\widehat{\mu}}$.}

Recall that 
$$
\pd(z) = {\widehat{\mu}}\big(\{\Lambda \in {\widehat{X}} : \Delta(\Lambda) \ge z\}\big) = {\widehat{\mu}}\big(\{\Lambda \in X : \Lambda \cap B(0,e^z) = \varnothing\}\big),
$$
where for $\vv\in\R^k$ and $r \ge 0$ we let $B(\vv,r)$ be the open ball in $\R^k$ centered at $\vv$ of radius $r$ with respect to the supremum norm. It will be convenient to write $${S_r}:=\Delta^{-1}\big([\log r,\infty)\big) = \{\Lambda \in {\widehat{X}}: B(0,r) \cap \Lambda= \varnothing \}.$$
Our goal is thus to prove that \eq{ddlwiths}{
{cr^{-k}\leq {\widehat{\mu}}(S_r) \leq Cr^{-k} 
\text{ for all }r\geq 1, }}
where  $c,C$ are constants dependent only on $k$.}

\smallskip
First let us discuss the  upper bound. It is in fact a special case of a recent result of Athreya, namely a random Minkowski-type theorem for the space of grids \cite[Theorem~1]{Athreya}:

\begin{prop}[Athreya]\label{Athreya} For a measurable $E\subset \mathbb{R}^k$,
$${\widehat{\mu}}\big(\{\Lambda \in {\widehat{X}}: \Lambda \cap E = \varnothing \} \big)\leq \frac{1}{1+\lambda (E)}.$$
\end{prop}
{\noindent}{Here} and hereafter $\lambda$ stands for Lebesgue measure on $\R^k$.
Taking $E= B(0,r)$ {shows that} ${\widehat{\mu}}(S_r) < {2^{-k}r^{-k}}$. 
Thus it only remains to establish a lower bound in \equ{ddlwiths}.


\smallskip
There exists an obvious projection, $\pi: {\widehat{X}} \rightarrow {X}$, making ${\widehat{X}}$ into a $\mathbb{T}^k$-bundle over ${X}$ \linebreak ($\pi$ simply translates  {one of the vectors in a grid} to the origin). 
{It is easy to see that  $\mu_{{\widehat{G}}}$ is the 
product of $\mu_{G}$ and  $\lambda$. Therefore 
one has the following Fubini formula:
{\eq{fubini}{{\widehat{\mu}} (S_{r})=\int_{{X}} 
Q(\Lambda, r) \,d\mu(\Lambda),\quad \text{ where }Q(\Lambda, r):=\lambda\big(S_r\cap \pi^{-1}(\Lambda)\big).
}} Here, {for $\Lambda \in X$,} $\pi^{-1}(\Lambda)$ is identified with $\mathbb{R}^k/\Lambda$ via \eq{ident}{[\vv]\in \mathbb{R}^k/\Lambda \longleftrightarrow\Lambda - \vv,}
and, in the hope that it will not cause any confusion, we will let $\lambda$ 
stand for the normalized Haar measure on $\R^k/\Lambda$ for any $\Lambda\in {X}$.} 
{Writing $\rho_\Lambda$ for the projection  $\R^k\rightarrow \R^k/\Lambda$, we have
\begin{equation}\label{torusvolume}
\begin{aligned}
S(r)\cap \pi^{-1}(\Lambda)&=\big\{[\vv] \in \mathbb{R}^k/\Lambda: B(0,r)\cap (\Lambda - \vv)=\varnothing \big\}\\
&=\big\{[\vv] \in \mathbb{R}^k/\Lambda : B(\vv, r) \cap \Lambda =\varnothing\big\}= \rho_\Lambda\big(\R^k\ssm\bigcup_{\vv \in \Lambda}B(\vv, r)\big),
\end{aligned}
\end{equation} 
so that 
$Q(\Lambda, r)$ is the area of a region in a fundamental domain (parallelepiped) in $\R^k$ for $\Lambda$ consisting of points which are farther than $r$ from its vertices, that is, from all points of $\Lambda$.}

\smallskip
 
Recall that
$\SL_2(\mathbb{R})$ double-covers the unit tangent bundle of the hyperbolic upper-half plane, $\mathbb{H}^2$. Since the action of $\SL_2(\mathbb{Z})$ on $\mathbb{H}^2$ has a convenient fundamental domain, there are convenient coordinates for a set of full measure in $\SL_2(\mathbb{R})/\SL_2(\mathbb{Z})$. This enables us to give a rather {tidy proof} for the two-dimensional case of \equ{ddlwiths}, {
handling both bounds simultaneously without {using} Proposition \ref{Athreya}}. This proof also illustrates the main idea necessary to proving the lower bound in the general case. We therefore start with a separate, redundant proof of the two-dimensional case.


\begin{proof}[Proof of \equ{ddlwiths} for $k=2$] 
\ignore{With the identification 
\equ{ident}, 
and writing $\rho_\Lambda$ for the projection  $\R^2\rightarrow \R^2/\Lambda$, we have
\begin{equation}\label{torusvolume}
\begin{aligned}
S(r)\cap \pi^{-1}(\Lambda)&=\big\{[\vv] \in \mathbb{R}^2/\Lambda: B(0,r)\cap (\Lambda - \vv)=\varnothing \big\}\\
&=\big\{[\vv] \in \mathbb{R}^2/\Lambda : B(\vv, r) \cap \Lambda =\varnothing\big\}= \rho_\Lambda\big(\R^2\ssm\bigcup_{\vv \in \Lambda}B(\vv, r)\big)
\end{aligned}
\end{equation} 
so that \eq{defQ}{Q(\Lambda, r):= \lambda\big(S(r)\cap \pi^{-1}(\Lambda)\big)} is the area of a region in a fundamental parallelogram consisting of points which are farther than $r$ from its vertices, that is, from all points of $\Lambda$.}
For fixed $r$, consider the map $({\kappa},{n},{a})\mapsto Q({\kappa}{n}{a} \Z^2, r)$, whose domain is ${K}\times {N}\times {A}$, the Iwasawa decomposition   for $G = \SL_2(\R)$. (Here ${K}$, $N$, $A$ are the groups of orthogonal, upper-triangular unipotent, and diagonal matrices respectively.) We first show that a change of ${\kappa}$ does not significantly change the value of $Q$. Indeed, since rotation perturbs the sup norm by no more than a factor of $\sqrt{2}$, for any ${\kappa} \in {K}$ and $\x,\y \in \mathbb{R}^2$ we have:
$$\|{\kappa}\x- {\kappa}\y\| \geq \sqrt{2}r \implies \|\x- \y\| \geq r \implies \|{\kappa}\x- {\kappa}\y\| \geq  {r}/{\sqrt{2}}\,,$$ 
hence
$$\rho_\Lambda \big(\mathbb{R}^2\ssm\bigcup_{\vv \in \Lambda}B({\kappa}\vv, \sqrt{2}r)\big)\subset \rho_\Lambda \big(\mathbb{R}^2\ssm\bigcup_{\vv \in \Lambda}B(\vv, r)\big)\subset \rho_\Lambda \big(\mathbb{R}^2\ssm\bigcup_{\vv \in \Lambda}B({\kappa}\vv, {r}/{\sqrt{2}})\big).$$ By \eqref{torusvolume}, this implies \eq{estimate}{ Q({\kappa}\Lambda, \sqrt{2}r)\leq Q(\Lambda, r)\leq Q({\kappa}\Lambda,  {r}/{\sqrt{2}}).}

Let ${a}= \diag({\alpha},{\alpha}^{-1})$. If \eq{2r}{{2r > {\alpha},}} then the lattice ${n}{a}\Z^2$ consists of horizontal rows of vectors, each closer than $2r$ to its horizontal neighbors. Thus the boxes making up the union $\cup_{\vv \in {n}{a}\Z^2}
{B(\vv, r)}$ overlap in the horizontal direction, creating horizontal strips. Thus, by \eqref{torusvolume}, $Q({n}{a}\Z^2, r)$ is just the area of the fundamental parallelogram ${n}{a}(I\times I)$ minus the strips on top and bottom, as in the following figure.

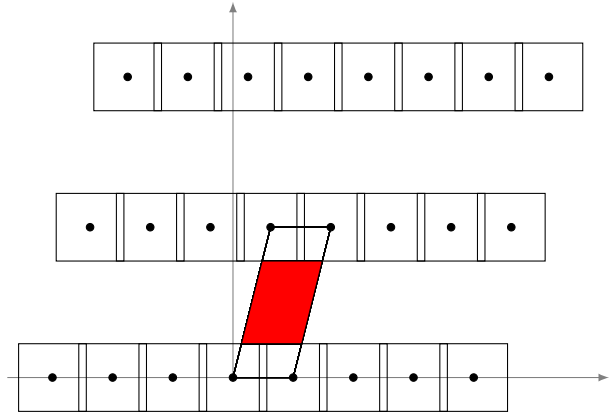
\begin{center} \begin{tikzpicture}

\coordinate (Origin) at (0,0);
\coordinate (XAxisMin) at (-3,0);
\coordinate (XAxisMax) at (5,0);
\coordinate (YAxisMin) at (0,0);
\coordinate (YAxisMax) at (0,5);
\draw [thin, gray,-latex] (XAxisMin) -- (XAxisMax);
\draw [thin, gray,-latex] (YAxisMin) -- (YAxisMax);
\foreach \x in {-3,...,4}{
\foreach \y in {0,...,2}{
\pgfmathsetmacro{\xval}{.8*\x+.5*\y}
\pgfmathsetmacro{\yval}{2*\y}
\draw (\xval-.45,\yval-.45) rectangle (\xval+.45,\yval+.45);

}
}

\pgftransformcm{.8}{0}{.5}{2}{\pgfpoint{0cm}{0cm}}

\foreach \x in {-3,-2,...,4}{
\foreach \y in {0,...,2}{
\node[draw,circle,inner sep=1pt,fill] at (\x,\y) {};

\draw (0,0) rectangle (1,1);

\draw [fill=red] (0,.225) rectangle (1,1-.225);

}
}

\end{tikzpicture} \end{center}

This smaller parallelogram has area ${\alpha}({\alpha}^{-1}-2r)$, provided it is nonempty, i.e.\ provided $2r\leq {\alpha}^{-1}$. {Thus from \equ{estimate}, if $\Lambda={\kappa}{n}{a}\Z^2$, where $a = \diag(\alpha, \alpha^{-1})$, we have
$$   Q(\Lambda, r) \leq Q(\kappa^{-1}\Lambda, r/\sqrt{2}) =  Q(na\Z^2, r/\sqrt{2}). $$
Then if ${\sqrt{2}r > {\alpha}}$ (so that \equ{2r} holds for $na\Z^2$, after adjusting $r$ as above), we have}
\begin{equation}\label{parallelogram}
\begin{aligned}
{\alpha}({\alpha}^{-1}-2\sqrt{2}r)\leq Q(\Lambda, r) \leq {\alpha}({\alpha}^{-1}-\sqrt{2}r) \qquad  \text{ if }&\sqrt{2}r\leq {\alpha}^{-1};\\
Q(\Lambda,r)=0  \qquad \text{ if }&\sqrt{2}r\geq {\alpha}^{-1}.
\end{aligned}
\end{equation} 

We now identify $X_2$ with ${\SL_2(\Z)}\backslash T^1(\mathbb{H}^2)$ via ${g}\Z^2\mapsto {g}^{-1}(i,i)$ (here, the matrix ${g}^{-1}$ acts on $(i,i)$ as a fractional-linear  transformation). Recall that $${F}:=T^1\{|z|\geq 1,\, |\re(z)|\leq 1/2\}$$ is a fundamental domain for the action of $\SL_2(\mathbb{Z})$ on $T^1(\mathbb{H}^2).$ Under the correspondence $ {g}^{-1}(i,i)\mapsto {g}\Z^2$, a point $(z,\theta) \in {F}$ maps to a lattice $\Lambda = {\kappa}{n}{a}\Z^2$ with $\sqrt{3}/2\leq \im(z) = \alpha^{-2}$. 
Thus if  {\eq{newbound}{\sqrt{3}r^2  > 1} (which  ensures that the condition $\sqrt2r > \alpha$ for (\ref{parallelogram}) is met)} 
then the estimates \eqref{parallelogram} hold with $\im(z) = \alpha^{-2}$ for any lattice $\Lambda_{(z,\theta)}$, $(z,\theta)\in {F}$.
Writing $y= \im(z)$, the estimates become

\begin{equation}\label{estimate2}
\begin{aligned}
1-\frac{2\sqrt{2}r}{\sqrt{y}}\leq Q(\Lambda_{(z,\theta)}, r) \leq 1-\frac{\sqrt{2}r}{\sqrt{y}} \qquad  \text{ if }&\sqrt{2}r\leq \sqrt{y};\\
Q(\Lambda_{(z,\theta)},r)=0  \qquad \text{ if }&\sqrt{2}r\geq \sqrt{y}.
\end{aligned}
\end{equation} 


\noindent Since the Haar measure on ${X_2}$ corresponds to the hyperbolic measure $\frac{1}{y^2} \,dxdyd\theta$ on ${F}$, we have
$$\int_{X_2'}Q(\Lambda, r)\, d\mu'(\Lambda) = \int_{F} Q(\Lambda_{(x+iy,\theta)}, r)\cdot \frac{1}{y^2} \,dxdyd\theta.$$

Finally, since $r$ is large\footnote{{Specifically we need $\sqrt{2}r\geq 1$, which is already covered by \equ{newbound}.}},
$Q(\Lambda_{(z,\theta)},r)$ vanishes in the region of $F$ between the line $y = 2r^2$ and the arc of the unit circle, permitting us to integrate over an unbounded rectangular region. The estimates \eqref{estimate2} give
$$2\pi \int_{8r^2} ^\infty\left(1-\frac{2\sqrt{2}r}{\sqrt{y}}\right)\frac{dy}{y^2} \leq \int_{{X_2}} Q(\Lambda, r)\,d\mu(\Lambda) \leq 2\pi \int_{2r^2}^\infty\left(1-\frac{\sqrt{2}r}{\sqrt{y}}\right)\frac{dy}{y^2},$$
where the $2\pi$ comes from integrating a constant function over the $\theta$ factor. 
Computing these integrals gives $$\frac{\pi}{12r^2}
\leq {\widehat{\mu}}(S_r) \leq \frac{\pi}{3r^2}
\hspace{3mm} whenever \hspace{3mm}  r \geq 3^{-1/4}.$$ This proves \equ{ddlwiths}.
%
 %
\end{proof}

\ignore{We will now prove {Theorem \ref{thmDL}} for arbitary $k$. Our upper bound is a special case of a result of Jayadev Athreya [A, Theorem 1].

\begin{prop}[Athreya]\label{Athreya} For measurable $A\subset \mathbb{R}^k$,

$$\mu_{X_k}\{\Lambda \in X: \Lambda \cap A = \varnothing \} \leq \frac{1}{1+\lambda (A)}$$

\noindent where $\lambda$ is Lebesgue measure. $\square$

\end{prop}

\noindent Taking $A:= B(0,r)$ gives $\Phi_\Delta(r)\leq 2e^{-kr}$. \comm{I am not sure it is correct, shouldn't $r$ be replaced by $e^r$?} We will obtain a lower bound in a way similar to the proof of the two-dimensional case.

\begin{prop}\label{lowerbound} There exists $c>0$ such that $ce^{-kr}\leq \Phi_\Delta(r)$ for large $r$. \comm{Need to mention that it would imply the same for all $r \ge 0$, or else choose a large $z_0$ in  Theorem \ref{thmDL}.} \end{prop}
\noindent }


\section{{Completion of the} proof of Theorem \ref{thmDL}}\label{deltadl}

\noindent We now set up the proof of the general case {($k \ge 2$)} with some notation and remarks on Siegel sets. Then the proof will be given following two lemmas generalizing some statements from the proof of the two-dimensional case.



As before, we wish to write $Q(\Lambda, r)$ introduced in \equ{fubini} in terms of the coordinates of the Iwasawa decomposition of a representative  ${g}\in {G}
$  for $\Lambda={g}\mathbb{Z}^k$. We will assume ${g}$ lies in a subset of a particular Siegel set. Specifically, for elements of ${G}
$ of the form
\eq{aandn}{{n} = \begin{bmatrix}
1 & \nu_{1,1} & \nu_{1,2} & \cdots & \nu_{1, k-1} \\
0 & 1 & \nu_{2,1} & \cdots & \nu_{2,k-2} \\
0 & 0 & 1 & \ddots & \vdots \\

\vdots & \vdots & \ddots & \ddots & \nu_{k-1,1} \\

0 & 0 & \cdots & 0 & 1 \\

\end{bmatrix}, \hspace{4mm} {a}= \begin{bmatrix}
a_1 & 0 & 0 & \cdots & 0 \\
0 & a_2 & 0& \cdots & 0 \\
0 & 0 & a_3 & \ddots & \vdots \\

\vdots & \vdots & \ddots & \ddots & 0 \\

0 & 0 & \cdots & 0 & a_k \\

\end{bmatrix},}

\noindent and for $d, e\in \mathbb{R}$, $c \in \mathbb{R}_+$, define $$\ca_c := \{{a} \in \ca: a_{j+1}\geq ca_j>0 \hspace{2mm} (j=1,..., k-1)\},$$ $$ N_{e,d}:= \{ {n} \in  N: e\leq \nu_{i, j} \leq d \hspace{2mm} (1\leq i, j \leq k-1)\}.$$ Also write ${K}$ for $\SO{(k)}$. It is known that $ K  A_{1/2} N_{-1,0}$ is a ``coarse fundamental domain" for {$\Gamma_k$ in $G_k$ 
(see}\ \cite[\S 19.4(ii), following Remark 7.3.4]{Mo}\footnote{Our definition is that of \cite{Mo} post-composed with $g\mapsto g^{-1}$, since our action is on the right.}). {That is}, $ K  A_{1/2} N_{-1,0}$ contains a fundamental domain for the right-action of {$\Gamma_k$ on $G_k$},
and it is covered by finitely many {$\Gamma_k$}-translates
of that domain.
Therefore $ K  A_1 N_{-1,0}$ is contained in a coarse fundamental domain, and since we are interested in a lower bound for $\int_{{X}} Q(\Lambda, r) \,d\mu(\Lambda)$, it will suffice to bound the integral \begin{equation}\label{integral}\int_{ K  A_1 N_{-1,0}}Q({g}\mathbb{Z}^k, r) \,d{\mu_{G}(g)}\end{equation} 
from below.

For the purpose of the lower bound it will suffice to restrict ourselves to 
{the} subset of $ K  A_1 N_{-1,0}$ {with $a$ satisfying
\eq{order}{0<a_1\leq a_2\leq  \dots \leq a_{k-1} < 2r\leq a_k;}
as we will show, the integral over this set contains the highest order term of \eqref{integral} as a function of $r$.}

\begin{lem}\label{parallelipiped} Suppose  {$a$ and $n$ are as in \equ{aandn}, and assume that $a$ 
satisfies \equ{order}}.
Then
\begin{equation}\label{Q} Q({n}{a}\mathbb{Z}^k, r) = 1-2ra_1...a_{k-1}. \end{equation}
\end{lem}

\begin{proof} The proof follows that of the two-dimensional case. Write $\Lambda = {n}{a} \mathbb{Z}^k$ and let \linebreak  $\rho_\Lambda: \mathbb{R}^k \rightarrow \mathbb{R}^k/\Lambda$ be the projection. 
Using \eqref{torusvolume}, one can write
\eq{Q2}{
Q(\Lambda, r) 
=
{\lambda} \Big(\rho_\Lambda \big(\mathbb{R}^k {\ssm} \bigcup_{\vv \in \Lambda} B(\vv, r)\big)\Big)= \lambda \Big ({n}{a}I^k{\ssm}\bigcup_{\vv \in \Lambda} B(\vv, r)\Big) ,
}
where $I^k=[0,1]\times ...\times [0,1]$, and $\lambda$, {as before, stands for both the normalized volume on $\pi^{-1}(\Lambda)$ and} Lebesgue measure on $\mathbb{R}^k$. 
{\equ{order}} implies
$$\bigcup_{\vv\in \Lambda} B(\vv,r) = \mathbb{R}^{k-1} \times \bigcup_{{\ell}\in \mathbb{Z}} ({\ell}a_k-r, {\ell}a_k+r),$$
so that the measure of ${n}{a}I^k {\smallsetminus} \cup_ {\vv \in \Lambda} B(\vv,r)$ is the measure of a parallelepiped of dimensions $a_1, a_2, ..., a_{k-1}$ and $ a_k- 2r$, precisely as in the two-dimensional case. In fact the figure used in the proof of the two-dimensional case is still illustrative: just replace the squares with hypercubes, let the $y$ axis stand for the $a_k$ axis, and let the $x$ axis stand for the hyperplane $a_k=0$. This yields \eqref{Q}. \end{proof}

The next lemma will allow us to disregard the factor ${K}$ when estimating the integral \eqref{integral}.

\begin{lem}\label{K} For ${\kappa}\in K
$ and $\Lambda \in {X}$, $$ Q(\Lambda, k^{1/2}r) \leq Q({\kappa}\Lambda, r)\leq Q(\Lambda, k^{-1/2}r).$$\end{lem}

\begin{proof} If ${P}\subset \mathbb{R}^k$ is a fundamental parallelipiped for the action of $\Lambda$ on $\R^k$, \equ{Q2} gives $$Q({\kappa}\Lambda, r)= \lambda\big({\kappa}{P}{\ssm}\cup_{\vv\in \Lambda}\{B(0,r)+{\kappa}\vv\}\big)= \lambda\big({P}{\ssm}\cup_{\vv\in \Lambda}\{{\kappa}^{-1}B(0,r)+\vv\}\big).$$ But $$B(0, rk^{-1/2})\subset {\kappa}^{-1}B(0,r)\subset B(0, rk^{1/2}),$$ so the result follows from another application of \equ{Q2}. \end{proof}

{Now we are ready to write down the}
\begin{proof}[{Proof of \equ{ddlwiths} for $k>2$}] 
Let $da, dn, d\kappa$ denote Haar measures on $\ca$, $ N$, and $K$.
Define $$\eta:  A \rightarrow \mathbb{R},\ \ \ {a}=\diag(a_1,...a_k) \mapsto \Pi_{i<j}\frac{a_i}{a_j}.$$ Then the Iwasawa decomposition identifies {$\mu_{G}$}
with the product measure $\eta({{a}})\,d\kappa \hspace{1mm} da\hspace{1mm} dn$  \linebreak (cf.~\cite[V.2.4]{Bekka-Mayer}).  Recall that we aim to bound the integral \eqref{integral}  from below. {Let us write $n^a = ana^{-1}$ for $n \in N$, $a \in A$.} By decomposing {$\mu_{G}$}
as above and restricting the domain of integration, we have

{\begin{equation*}
\begin{aligned}
\int_{ K  A_1 N_{-1,0}}Q({g}\mathbb{Z}^k, r) \,d{\mu_{G}(g)}  &= \int_{ K  A_1 N_{-1,0}}Q(\kappa a n\mathbb{Z}^k, r) \,d\kappa\, da\, dn  \\ &= \int_{ K  A_1 N_{-1,0}}Q(\kappa n^a a\mathbb{Z}^k, r) \,d\kappa\,da\,dn \geq    \\ &=\int_ K \int_{ N_{-1,0}} \int_{\{{a}\in  A_1: a_{k-1}\leq 2r\sqrt{k} \leq a_k\}} Q({\kappa}{n^a}{a}\Z^k, r)  \eta({a})\, da \,dn\, d\kappa. \end{aligned}\end{equation*}}
By Lemma \ref{K}, this latter integral is not smaller than $$\int_ K \int_{ N_{-1,0}} \int_{\{{a}\in  A_1: a_{k-1}\leq  2r\sqrt{k} \leq a_k\}} Q({n^a}{a}\Z^k, k^{1/2}r)  \eta({a})\, da \,dn\, d\kappa,$$ 
and by Lemma \ref{parallelipiped} this is the same as $$\int_ K \int_{ N_{-1,0}} \int_{\{{a}\in  A_1: a_{k-1}\leq 2r\sqrt{k} \leq a_k\}}\left(1-2rk^{1/2}a_1...a_{k-1}\right)\eta({a})\, da \,dn\, d\kappa.$$
Since this integrand depends only on ${a}$, and the other factors have finite measure, it suffices to consider $$\int_{\{{a}\in  A_1: a_{k-1}\leq 2r\sqrt{k} \leq a_k\}}\left(1-2rk^{1/2}a_1...a_{k-1}\right)\eta({a})\, da. $$ 

Finally we identify $da$ with Lebesgue measure (up to a constant) on $\mathbb{R}^{k-1}$ via $$\diag(a_1,...a_k)\mapsto \big(\log (a_1), \log (a_2),..., \log (a_{k-1})\big),$$ see \cite[V.2.3]{Bekka-Mayer}\footnote{Our identification is theirs composed with a linear isomorphism of $\mathbb{R}^{k-1}$.}. We are therefore left with the integral $$\int_{b_1\leq b_2\leq ... \leq b_{k-1} \leq  \log(2r\sqrt{k}) \leq -\Sigma_{i=1}^{k-1}b_i}\left(1-2rk^{1/2}\exp[\Sigma_{i=1}^{k-1}b_i]\right)\exp[\Sigma_{i<j}(b_i-b_j)] \,d\lambda,$$ where the $b_k$'s occurring in the exponent of the second factor of the integrand must be understood to stand for $-\Sigma_{i=1}^{k-1}b_i$.

  Now the challenge is not the integrand (which consists  of nice exponential functions) but the domain of integration. Thankfully we only have to integrate over a piece of it, since we are interested in a lower bound. The piece we will consider is the following set: \eq{S}{ \left\{(b_1,...b_{k-1}): b_i\leq b_{i+1}\leq \frac{-\log(2r\sqrt{k})}{k-1} ~~ (1\leq i\leq k-2)~~ \right\}.} This set is clearly contained in the domain of integration above. Reordering the variables $x_i:= b_{k-i}$, and using the identity $\sum_{i<j}x_i-x_j= \sum_{i=1}^{k-1} 2ix_i$, we can compute the integral of $Q({a}\mathbb{Z}^k, \sqrt{k}r)$ over \equ{S} as an iterated integral: \begin{equation}\label{int} \int_{-\infty}^{\frac{-\log (2r\sqrt{k})}{k-1}}\int_{x_{k-1}}^{\frac{-\log (2r\sqrt{k})}{k-1}}
  \cdots \int_{x_{2}}^{\frac{-\log (2r\sqrt{k})}{k-1}} \left(e^{\sum_{i=1}^{k-1} 2ix_i}-2re^{\sum_{i=1}^{k-1} (2i+1)x_i}\right) \,dx_1\,dx_2\cdots dx_{k-1}.\end{equation}  It is easily seen by induction that for $2\leq {\ell}\leq k-1$, $$\int_{x_{\ell}}^{\frac{-\log (2r\sqrt{k} )}{k-1}}\int_{x_{{\ell}-1}}^{\frac{-\log (2r\sqrt{k})}{k-1}} \cdots\int_{x_{2}}^{\frac{-\log (2r\sqrt{k})}{k-1}} \left(e^{\sum_{i=1}^{k-1} 2ix_i}-2re^{\sum_{i=1}^{k-1} (2i+1)x_i}\right) \, dx_1\,dx_2 \cdots dx_{{\ell}-1}$$ is a sum of terms of the form $$c(2r\sqrt{k})^{-m/(k-1)}e^{\sum_{i={\ell}}^{k-1} p_ix_i}$$ where $c\ {> 0}$, $p_i$ are positive integers, and $m+\sum_{i={\ell}}^{k-1} p_i= k(k-1)$. Indeed
$$\int_{x_{{\ell}+1}}^{-\log(2r\sqrt{k})/(k-1)}c(2r\sqrt{k})^{-m/(k-1)}e^{\sum_{i={\ell}}^{k-1}p_ix_i}\,dx_{\ell}=$$

$$ \frac{c}{p_{\ell}}(2r\sqrt{k})^{\frac{-(m+p_{\ell})}{k-1}}\exp\left[\sum_{i={\ell}+1}^{k-1}p_ix_i\right]-\frac{c}{p_{\ell}}(2r\sqrt{k})^\frac{-m}{{k-1}}\exp\left[(p_{\ell}+p_{{\ell}+1})x_{{\ell}+1}+\sum_{i={{\ell}+2}}^{k-1} p_ix_i\right],$$
so that we have only to notice that $$(m+p_{\ell})+\sum_{i={\ell}+1}^{k-1}p_i= m+[(p_{\ell}+p_{{\ell}+1})+\sum_{i={{\ell}+2}}^{k-1}p_i]=m+\sum_{i={\ell}}^{k-1}p_i =k(k-1)$$ from the induction hypothesis. Thus $\eqref{int} $ is a sum of terms of the form
\begin{equation*}
\begin{aligned}
\int_{-\infty}^{\frac{-\log(2r\sqrt{k})}{k-1}} c(2r\sqrt{k})^{-m/(k-1)}e^{p_{k-1}x_{k-1}} \,dx_{k-1} &= \frac{c}{p_{k-1}}(2r\sqrt{k})^{\frac{-(m+p_{k-1})}{k-1}} \\ & =\frac{c}{p_{k-1}}(2r\sqrt{k})^{\frac{-k(k-1)}{k-1}} =\frac{c}{p_{k-1}}(2r\sqrt{k})^{-k}, \end{aligned}\end{equation*}
where we have used $m+p_{k-1}=k(k-1)$. Since the integral is positive, the sum of the coefficients must be positive, and the integral grows no more slowly than some multiple of $r^{-k}$. 
\end{proof}

\section{{Concluding remarks and open questions}}
\subsection{{The homogeneous problem}}
Here we return to the homogeneous case and discuss the approach to Question \ref{dirichlet question} suggested by the foregoing argument.  Recall that $X=X_k = \SL_k(\R)/\SL_k(\Z)$ is the space of unimodular lattices in $\R^k$.  Define $$\Delta_0:X_k \rightarrow \R,\ \  \Lambda \mapsto \log\inf_{\vv\in \Lambda \ssm 0} \|\vv\|,$$ and for $A \in M_{m,n}$, define 
$$\Lambda_A:=  \left( \begin{array}{cc} I_m & A \\
0 & I_n \\ \end{array} \right) \mathbb{Z}^{m+n}{\in X_{{k}}},$$
{where $k = m+n$.}
If we restrict the flow $g_t$ to $X_k$, it is not difficult to show\footnote{See \cite[Proposition 4.5]{KW}, though notice {that the function used there differs from $\Delta_0$ by a minus sign.}
} the following homogeneous version of Lemma \ref{inhomogeneousshrinkingtarget}:

\begin{prop}
Fix positive integers $m, n$, and let $\psi: [t_0, \infty)\rightarrow (0,1)$ be continuous and non-increasing. Let $z=z_\psi$ be as in Lemma \ref{dani1}. Then $A\in D_{m,n}(\psi)$ if and only if $$\Delta_0 ({g}_s\Lambda_A)<  z_\psi(s)$$ for all sufficiently large $s$. \end{prop}

This {way Question  \ref{dirichlet question} reduces to}
a shrinking target problem {for the flow $(X,g_t)$}, 
where the
targets are 
super-level sets $\Delta_0^{-1}\big([z,\infty)\big)$.  But the family of super-level sets {of $\Delta_0$} differs in important ways from the family of super-level sets {of $\Delta$}.  In particular, by Minkowski's Theorem, $\Delta_0^{-1}[z,\infty)$ is empty for $z> 0$.   {Hence the problem reduces to  the case}
where the values $z_\psi(t)$ accumulate at $0$, so that the targets shrink  to the set $\Delta_0^{-1}(0)$.  The {latter} set 
is a union of finitely many compact submanifolds of $X$ whose structure is explicitly described by the Haj\'os-Minkowski Theorem (see \cite[\S XI.1.3]{CasselsGN} or \cite[Theorem 2.3]{Shah}). {In particular, the function $\Delta_0$} is not {DL}, and 
{Theorem \ref{thm4.3}}
is not applicable.   
{Other approaches to shrinking target problems on homogeneous spaces \cite{Kelmer, KY, KZ, Mau}
also do not seem to be directly applicable.}

{On the other hand, the   one-dimensional case ($m = n = 1$) has been completely settled in \cite{KW}. In particular, the following zero-one law has been established:
\begin{thm}[{\cite[Theorem 1.8]{KW}}] \label{Main Theorem} Let $\psi: [t_0,\infty) \rightarrow \R_+$ be non-increasing, and suppose the function $t\mapsto t\psi(t)$ is non-decreasing {and \eq{lessthan1}{t\psi(t) < 1\quad\text{for all }t\ge {t_0}.}}  Then if \eq{condition}{
{\sum_i\frac{-\log\big(1-i\psi(i)\big)\big(1-i\psi(i)\big)}{i}}
= \infty \hspace{3mm}(\text{resp. } <\infty),}
then the Lebesgue measure of $D_{1,1}(\psi)$ (resp.\  of $D_{1,1}(\psi)^c$) is zero.\end{thm}
The proof is based on the observation that the  condition $\alpha \in D_{1,1}(\psi)$ can be explicitly described in terms of the continued fraction expansion of $\alpha$. However, this phenomenon is inherently one-dimensional, and new ideas are needed  to settle the general case}.

\subsection{{\hd}} {A sequel \cite{HKWW} to the paper \cite{KW} computes the \hd\ of limsup sets $D_{1,1}(\psi)^c$, and, more generally, establishes zero-infinity laws for the Hausdorff measure of those sets. For example, it is proved there that 
$$
\dim \big(D(\psi)^c\big)=\frac{2}{2+\tau}\    \ {\text{when}}\ \psi(t)=\frac{1-at^{-\tau}}{t}  \ (a>0, \tau>0).
$$
One can ask similar questions for higher-dimensional versions, both in homogeneous and inhomogeneous settings. Even the $m = n = 1$ case of the inhomogeneous problem is open.
}

\subsection{{Singly vs.\ doubly metric problems}} {The main result of the present paper computes Lebesgue measure of the set $\admn(\psi)\subset M_{m,n}\times \mathbb{R}^m$. 
As often happens in inhomogeneous \di\ problems, one can  fix either $A$ or $\vb$ and ask for  the Lebesgue (or Hausdorff) measure of the corresponding slices of $\admn(\psi)$. It seems plausible that the convergence/divergence of the same series \equ{psisum} is responsible for a full/zero measure dichotomy for slices $$\{A\in M_{m,n}: (A,\vb) \in \admn(\psi)\}$$ for any fixed $\vb\notin\Z^m$. On the other hand, the Lebesgue measure of the set $$\{\vb\in \R^m: (A,\vb) \in \admn(\psi)\}$$ for a fixed $A\in M_{m,n}$ seems to depend heavily on \di\ properties of $A$. For example, if $A$ has rational entries, then  $(A,\vb)$ is not in  $\admn(\psi)$ whenever $\vb\notin \Q^m$ and $\psi(T) \to 0$ as $T\to\infty$. And on  the other end of the approximation spectrum, if $A$ is badly approximable  it is easy to see that there exists $C > 0$ such that for all $\vb\in\R^m$, $(A,\vb)$ belongs to the (null) set $\admn(C\psi_1)$. Indeed, by the classical Dani Correspondence, $A$ is badly approximable  if and only if the trajectory $\{g_t\Lambda_A : t > 0\}$ is bounded in $X_k$, which is the case if and only if $\{g_t\Lambda_{A,\vb} : t > 0\}$ is bounded in $\widehat{X}_k$ for any $\vb\in\R^m$. Thus the claim follows in view of Lemma \ref{corr_c/t}.
It would be interesting to describe, for a given arbitrary non-increasing function $\psi$, explicit \di\ conditions on $A\in M_{m,n}$ guaranteeing that  $(A,\vb) \in \admn(\psi)$ for all (or almost all) $\vb\in\R^m$.}

\subsection{{Eventually always hitting}} {Finally, let us connect our  results on improving the inhomogeneous Dirichlet Theorem with a shrinking target property introduced recently by Kelmer \cite{Kelmer}. 
We start by setting some notation. Let $\alpha$ be a measure-preserving $\Z^n$-action  on a probability space $(Y,\nu)$. 
For any $N \in {\N}$ denote  $$D_N:= \{\vq\in\Z^n: \|\vq\| \le N\}$$ (here, as before,  $\|\cdot\|$ stands for the supremum norm). Then 
given a {nested} family
$\mathcal{B} = \{B_N: N\in \N\}$ 
of subsets of $Y$,
let us say that the  $\alpha$-orbit of a point
$x
\in Y$
{\sl eventually always hits} $\mathcal{B}$  if $\alpha(D_N)x\cap B_N \ne\varnothing$
for all sufficiently large
$N\,{\in\N}$. 
Following \cite{Kelmer}, denote by
$\mathcal{A}^\alpha_{\bf ah}(\mathcal{B})$ the set of points of $Y$ with
$\alpha$-orbits eventually always hitting $\mathcal{B}$. This is a liminf set with a rather complicated structure. In  \cite{Kelmer} 
sufficient conditions for sets $\mathcal{A}^\alpha_{\bf ah}(\mathcal{B})$ to be of full measure were found for unipotent and diagonalizable actions $\alpha$ on hyperbolic manifolds. Namely it was shown\footnote{{Note that Kelmer considered the eventually always hitting property for forward orbits, that is, with sets $D^+_N:= \{\vq\in\Z^n: q_i \ge 0, \|\vq\| \le N\}$ in place of $D_N$.}} (see \cite[Theorem 22 and Proposition 24]{Kelmer})  that for  rotation-invariant monotonically shrinking families  $\mathcal{B}$,
$\nu\big(\mathcal{A}^\alpha_{\bf ah}(\mathcal{B})\big) = 1$ if the series
\eq{dubissum}{
\sum_j\frac1{2^{nj}\nu(B_{2^j})}}
converges. 
See also \cite{KY} for some extensions to actions on  \hs s of semisimple Lie groups. However, to the best of  the authors' knowledge, there  are no nontrivial examples of measure-preserving systems for which necessary and sufficient conditions  for  sets $\mathcal{A}^\alpha_{\bf ah}(\mathcal{B})$ to be of full measure exist in the literature.}

{Now, given \amr, take $Y = \T^m$ with normalized Lebesgue measure $\nu$ and consider the $\Z^n$-action \eq{aalpha}{\x\mapsto \alpha(\vq)\x := \x + A\vq \mod \Z^m} on $Y$ (generated by $n$ independent rotations of $\T^m$ by the column vectors of $A$). Also fix $\y\in Y$ and a {non-increasing} 
sequence $\{r(N){: N\in \N}\}$ of positive numbers, and consider the family $\mathcal{B}$ of   open balls \eq{target}{B_N := \{\x\in \T^m: \|\x - \y\| < r(N)\}.}
Then {it is easy to see that}  
$\x\in \mathcal{A}^\alpha_{\bf ah}(\mathcal{B})$ if and only if for all sufficiently large $N\,{\in\N}$ there exist $\vq\in\Z^n$ and $\vp\in\Z^m$ such that 
\begin{equation}\label{eah}\|\vq\| \,{< N + 1}\qquad\text{ and }\qquad \|\x + A\vq - \vp - \y\| < r(N).\end{equation}
Here and hereafter $\alpha$ and $A$ are related via \equ{aalpha}. A connection to the improvement of inhomogeneous \dt\ is now straightforward. Indeed, from Theorem \ref{mainthm}  one can derive the following
\begin{cor}\label{coreah} {Fix $\y\in\T^m$ and} let $\mathcal{B} = \{B_N: N\in \N\}$ be as in \equ{target}, where $\{r(N): N\in\N\}$ is a non-increasing 
sequence of positive numbers. Then for Lebesgue-a.e.\ \amr\ the set $\mathcal{A}^\alpha_{\bf ah}(\mathcal{B})$ has zero (resp.\ full) measure provided the sum \equ{dubissum} diverges (resp\  converges).
\end{cor}}
\begin{proof}{
{Extend $r(\cdot)$ to a non-increasing continuous function on $\R_+$ in an arbitrary way (for example, piecewise-linearly). Then, similarly to the observation made after \equ{dtpsi}, one can notice that $\x\in \mathcal{A}^\alpha_{\bf ah}(\mathcal{B})$ if and only if the system \eqref{eah} is solvable in integers $\vp,\vq$ for all sufficiently large $N\in\R_+$. The latter happens if and only if}
the pair $(A,\x-\y)$ belongs to $\admn(\psi)$, where $$\psi(T) := r (T^{1/n}-1)^m.$$ 
In view of  Theorem \ref{mainthm}, the divergence of the sum \begin{equation*}\label{computation}\begin{split}\sum_j \frac{1}{\psi(j)j^2} &={\sum_j \frac{1}{r(j^{1/n}-1)^mj^2} \asymp\int\frac{dx}{r(x^{1/n}-1)^mx^2}  \asymp \int\frac{(y+1)^{n-1}\,dy}{r(y)^m(y+1)^{2n}}} \\ & {\asymp}    \int\frac{dy}{r(y)^my^{n+1}}  \asymp \int\frac{2^z\,dz}{r(2^z)^m2^{z(n+1)}}  \asymp \sum_j \frac{1}{r(2^j)^m2^{nj}}\asymp \sum_j\frac1{2^{nj}\nu(B_{2^j})}\end{split}\end{equation*}  implies that  $\admn(\psi)$ has measure zero. Hence for a.e.\  $A$  the set  $\mathcal{A}^\alpha_{\bf ah}(\mathcal{B})$ is null. {Similarly, the convergence of \equ{dubissum} implies that $\admn(\psi)$ is conull. Thus}}
{for Lebesgue-generic  $A$  the set  $\mathcal{A}^\alpha_{\bf ah}(\mathcal{B})$ has full measure.} 
 \end{proof}

\begin{acknowledgements}
The authors would like to thank Alexander Gorodnik, Dubi Kelmer and Shucheng Yu for helpful discussions, and the anonymous referee for useful comments.
\end{acknowledgements}

\end{document}